\documentclass[12pt]{amsart}

\usepackage{amsmath,amssymb,amscd,amsthm,graphicx,enumerate}

\usepackage{hyperref}
\hypersetup{breaklinks=true}

\numberwithin{equation}{section}
\theoremstyle{plain}

\makeatletter
\newtheorem*{rep@theorem}{\rep@title}
\newcommand{\newreptheorem}[2]{%
\newenvironment{rep#1}[1]{%
 \def\rep@title{#2 \ref{##1}}%
 \begin{rep@theorem}}%
 {\end{rep@theorem}}}
\makeatother

\newtheorem{theorem}[equation]{Theorem}
\newreptheorem{theorem}{Theorem}

\newtheorem{corollary}[equation]{Corollary}

\theoremstyle{remark}

\theoremstyle{definition}

\newcommand{\R}{\mathbb R}

\newcommand{\al}{\alpha}

\newcommand{\D}{\partial}

\newcommand{\Lap}{\Delta}

\def\XXint#1#2#3{{\setbox0=\hbox{$#1{#2#3}{\int}$}
     \vcenter{\hbox{$#2#3$}}\kern-.5\wd0}}

\begin{document}

\title[Inscribed radius under mean curvature flow]{On Brendle's estimate for the inscribed radius under mean curvature flow}
\author{Robert Haslhofer}
\author{Bruce Kleiner}

\date{\today}

\begin{abstract}
In a recent paper \cite{Bre13}, Brendle proved that the inscribed radius of closed embedded mean convex hypersurfaces moving by mean curvature flow is at least $\frac{1}{(1+\delta)H}$ at all points with $H\geq C(\delta,M_0)$. In this note, we give a shorter proof of Brendle's estimate, and of a more general result for $\al$-Andrews flows, based on our recent estimates from Haslhofer-Kleiner \cite{HK13}.
\end{abstract}

\maketitle

\section{Introduction}
Let $\{M_t\subset \R^{n+1}\}$ be a family of closed embedded mean convex hypersurfaces that evolve by mean curvature flow. In a recent paper \cite{Bre13}, Brendle considered the inscribed radius $r_{\mathrm{in}}(p,t)$ and the outer radius $r_{\mathrm{out}}(p,t)$, i.e. the maximal radius of interior respectively exterior balls tangent at $p\in M_t$, and proved sharp estimates for them:

\begin{theorem}[{Brendle \cite[Thm. 1, Thm. 2]{Bre13}}]\label{cor}
Let $\{M_t\subset \R^{n+1}\}$ be a smooth family of closed embedded mean convex hypersurfaces which evolve by the mean curvature flow. Then, given any constant $\delta > 0$
there exists a constant $C=C(\delta,M_0)<\infty$ such that
\begin{equation}
r_{\mathrm{in}}(p,t) \geq \frac{1}{(1 + \delta) H(p,t)}\qquad \mathrm{and} \qquad r_{\mathrm{out}}(p,t)\geq \frac{1}{\delta H(p,t)},
\end{equation}
whenever $H(p,t) \geq C$.
\end{theorem}

Brendle's proof is based on some quite sophisticated computations. The purpose of this note is to give a shorter proof of Brendle's estimate.

In fact, we prove a more general theorem (Theorem \ref{main_thm}), that also applies to certain local and weak solutions.
To describe our setting, recall that a mean curvature flow of mean convex hypersurfaces $M_t=\partial K_t$ satisfies the $\alpha$-Andrews condition \cite[Def. 1.1]{HK13}, if each boundary point $p\in\partial K_t$ admits interior and exterior balls tangent at $p$ of radius at least $\frac{\alpha}{H(p,t)}$; in other words if $\inf\{Hr_\mathrm{in},Hr_\mathrm{out}\}\geq \alpha$.
By the main theorem of Andrews \cite{And13} (which we extended to weak solutions in \cite[Thm. 1.5]{HK13}), the Andrews condition is preserved under the flow,
i.e. if $K_0$ satisfies the $\alpha$-Andrews condition, then so does $K_t$ for all $t\geq 0$.
The class of $\alpha$-Andrews flows \cite[Def. 1.6]{HK13}
is then defined as the smallest class of set flows
 which contains all compact $\al$-Andrews level set 
flows with smooth initial condition and all smooth $\al$-Andrews flows $\{K_t\subseteq U\}_{t\in I}$ in any open set $U\subseteq \R^{n+1}$, and which is closed under the operations of
restriction, parabolic rescaling, and passage to Hausdorff limits.

\begin{theorem}\label{main_thm}
For all $\delta>0$, $\al >0$, there exists $\eta=\eta(\delta,\al)<\infty$ with the following property.
If $\{K_t\}$ is an $\al $-Andrews flow in a parabolic ball  $P(p,t,\eta r)$ centered at a boundary point 
$p\in \D K_t$ with $H(p,t)\leq r^{-1}$, then
\begin{equation}
r_{\mathrm{in}}(p,t)\geq \frac{r}{1+\delta} \qquad \mathrm{and} \qquad r_{\mathrm{out}}(p,t)\geq \frac{r}{\delta}.
\end{equation} 
\end{theorem}

 Note that Theorem \ref{main_thm} immediately implies Theorem \ref{cor}.
More generally, Theorem \ref{main_thm} shows that the scale invariant quantities $Hr_{\mathrm{in}}$ and $Hr_{\mathrm{out}}$ approach the optimal values $1$ respectively $\infty$, under the mere assumption that the flow is defined on a parabolic ball that is large enough compared to the curvature scale; the optimal value $1$ is attained for the shrinking cylinder $S^1\times\R^{n-1}$.
Philosophically, the assumption on the domain gives the maximum principle enough room to trigger in, and thus to improve the quantities $Hr_{\mathrm{in}}$ and $Hr_{\mathrm{out}}$ as much as allowed by the example of the cylinder. In particular, as another interesting consequence of Theorem \ref{main_thm}, we obtain:

\begin{corollary}\label{ancient}
Any ancient $\alpha$-Andrews flow $\{K_t\subset \R^{n+1}\}$ is in fact a $1$-Andrews flow.
\end{corollary}

Corollary \ref{ancient} shows that $\alpha=1$ is the universal noncollapsing constant for ancient mean convex mean curvature flows.
It shares some similarities with Perelman's theorem \cite[1.5]{Per03} that any ancient $\kappa$-solution for three-dimensional Ricci flow is either a $\kappa_0$-solution or a quotient of the round sphere; our constant is the only sharp one, though.

Our proof of Theorem \ref{main_thm} is related to our proof of the convexity estimate in \cite[Thm. 1.10]{HK13}, see also White \cite{Whi03}. The point is, once the relevant compactness theorem is established (see \cite[Thm. 1.8]{HK13} and \cite[Thm. 1.12]{HK13} respectively) we can pass to a suitable limit and apply the strong maximum principle.

\section{The proof}

\begin{proof}[Proof of Theorem \ref{main_thm}]
As explained in \cite[Sec. 4]{HK13} it suffices to establish the estimate for smooth $\al$-Andrews flows; the estimate for general $\al$-Andrews flows then follows by approximation.

Fix $\al<1$. The $\al$-Andrews condition implies that the assertion for the inscribed radius holds for $\delta = \frac1\al-1 $. 
Let $\delta_0\leq \frac1\al-1 $ be the infimum of the $\delta$'s for which it holds, 
and suppose $\delta_0>0$.

It follows that there is a sequence $\{K_t^j\}$ of $\al$-Andrews flows, 
where for all $j$, $0\in \D K_0^j$, $H(0,0)\leq 1$ and  $\{K_t^j\}$ is defined in 
$P(0,0,j)$, but $r_{\mathrm{in}}(0,0)\to \frac{1}{1+\delta_0}$ as $j\to \infty$.
By the global convergence theorem \cite[Thm. 1.12]{HK13}, after passing to a subsequence, 
$\{K_t^j\}$  converges smoothly and globally to an ancient convex $\alpha$-Andrews flow $\{K_t^\infty\subset \R^{n+1}\}$. Then for $\{K_t^\infty\}$ we have $r_{\mathrm{in}}(0,0)= \frac{1}{1+\delta_0}$ and $H(0,0)=1$.

By construction, the quantity $\frac{1}{Hr_{\mathrm{in}}}$ attains its maximum over $\partial K_t^\infty$ for all $t\leq 0$ at $(0,0)$. By Andrews-Langford-McCoy \cite{ALM13} we have the evolution inequality
\begin{equation}
\partial_t\frac{1}{Hr_{\mathrm{in}}}\leq\Lap\frac{1}{Hr_{\mathrm{in}}}+2\langle\nabla \log H,\nabla\frac{1}{Hr_{\mathrm{in}}}\rangle,
\end{equation}
in the viscosity sense. Thus, by the strict maximum principle, $Hr_{\mathrm{in}}$ is constant. On the one hand, this constant is equal to $\frac{1}{1+\delta_0}$. On the other hand, recalling that the asymptotic shrinkers of $\{K_t^\infty\}$  must be cylinders or spheres \cite[Rem. 1.20]{HK13}, we see that this constant must be at least $1$; this gives the desired contradiction and proves the estimate for the inscribed radius.

Finally, the estimate for the outer radius follows from the convexity of ancient $\alpha$-Andrews flows \cite[Cor. 2.13]{HK13}.
\end{proof}

\vspace{10mm}
{\sc Courant Institute of Mathematical Sciences, New York University, 251 Mercer Street, New York, NY 10012, USA}

\emph{E-mail:} robert.haslhofer@cims.nyu.edu, bkleiner@cims.nyu.edu

\end{document}